\documentclass[10pt,reqno]{amsart}
\usepackage{dsfont}
  \usepackage{paralist}
  \usepackage[final]{graphicx}
		\usepackage[usenames,dvipsnames]{color}
		\usepackage[usenames,dvipsnames,svgnames,table]{xcolor}
  \usepackage{epsfig}
 \usepackage[colorlinks=true]{hyperref}
 \usepackage{tikz}
\usepackage{stmaryrd}
\usepackage{framed}

\calclayout
\newtheorem{theorem}{Theorem}[section]

\newtheorem{lemma}[theorem]{Lemma}

\theoremstyle{definition}

\newtheorem{remark}{Remark}

\definecolor{eqboxcolor}{named}{Aquamarine}

\definecolor{lawboxcolor}{named}{Lavender}

\newcommand{\di}[1]{\,\mathrm{d}#1}
\newcommand{\jump}[1]{\llbracket #1\rrbracket}
\DeclareMathOperator*{\esssup}{ess\,sup}

\newcommand{\R}{\mathbb{R}}

\newcommand{\Q}{\mathbb{Q}}

\newcommand{\N}{\mathbb{N}}
\newcommand{\Z}{\mathbb{Z}}

\newcommand{\CC}{\mathcal{C}}
\newcommand{\per}{{\overline{\#}}}
\newcommand{\dive}{\operatorname{div}}
\newcommand{\dist}{\operatorname{dist}}

\newcommand{\e}{\varepsilon}
\newcommand{\p}{\varphi}

\newcommand{\dx}{\,\mathrm{d}x}
\newcommand{\ds}{\,\mathrm{d}s}

\newcommand{\dy}{\,\mathrm{d}y}

\newcommand{\recon}[1]{\left[#1\right]_\e}
\newcommand{\err}{\mathrm{err}}
\newcommand{\cor}{\mathrm{cor}}
\title[Corrector estimates for a thermoelasticity problem] 
      {Corrector estimates for the homogenization of a two-scale thermoelasticity problem with a priori known phase transformations}
\author[Michael Eden, Adrian Muntean]{}


 \email{leachim@math.uni-bremen.de}
 \email{adrian.muntean@kau.se}

\begin{document}
\maketitle

\centerline{\scshape Michael Eden}
\medskip
{\footnotesize
 \centerline{Center for Industrial Mathematics, FB 3}
   \centerline{University of Bremen, Germany}
   \centerline{ Bibliotheksstr.~1, 28359, Bremen}
} 

\medskip

\centerline{\scshape Adrian Muntean
}
\medskip
{\footnotesize
 \centerline{Department of Mathematics and Computer Science}
   \centerline{University of Karlstad, Sweden}
   \centerline{Universitetsgatan 2, 651 88 Karlstad}
}

\bigskip

\begin{abstract}
We investigate \emph{corrector estimates} for the solutions of a thermoelasticity problem posed in a highly heterogeneous two-phase medium  and its corresponding two-scale thermoelasticity model which was derived in~\cite{EM17} by two-scale convergence arguments.
The medium in question consists of a connected matrix with disconnected, initially periodically distributed inclusions separated by a sharp interface undergoing \emph{a priori} known phase transformations.
While such estimates seem not to be obtainable in the fully coupled setting, we show that for some simplified scenarios optimal convergence rates can be proven rigorously.
The main technique for the proofs are energy estimates using special reconstructions of two-scale functions and particular operator estimates for periodic functions with zero average.
Here, additional regularity results for the involved functions are necessary.
\end{abstract}
\section{Introduction}
We aim to derive quantitative estimates that show the quality of the upscaling process of a coupled linear  thermoelasticity system with a priori known phase transformations posed in a high-contrast media to its corresponding two-scale thermoelasticity system.


The problem we have in mind is posed in a medium where the two building components, initially assumed to be periodically distributed, are different solid phases of the same material in which phase transformations that are a priori known occur.
As a main effect, the presence of phase transformations leads to evolution problems in time-dependent domains that are not necessarily periodic anymore.

In our earlier paper~\cite{EM17}, we rigorously studied the well-posedness of such a thermoelasticity problem and conducted a homogenization procedure via the two-scale convergence technique (cf.~\cite{Al92} for details).
Those results were obtained after transforming the problem to a fixed reference geometry.
In this work, our goal is to further investigate the connection of those problems and to derive an upper bound for the convergence rate (in some yet to defined sense) of their solutions.
While historically a tool to also justify the homogenization (via asymptotic expansions) in the first place, such estimates, which in the homogenization literature are usually called \emph{error} and \emph{corrector estimates}, provide a means to evaluate the accuracy of the upscaled model.
Also, such estimates are especially interesting from a computational point of view.
In the context of \emph{Multiscale FEM}, for example, they are needed to ensure/control the convergence of the method, we refer to, e.g.,~\cite{AB05,HW97}.

The basic idea is to estimate the $L^2-$ and $H^1-$errors of the solutions of the problems using energy techniques, additional regularity results, and special operator estimates for functions with zero average. Since, in general, the solutions of the $\e$-problem do not even have the same domain as the solutions of the two-scale problem, we additionally rely on so-called \emph{macroscopic reconstructions} (we refer to Section~\ref{sec:preliminaries}).
The difficulty in getting such estimates in our specific scenario is twofold: First, the coupling between the quasi-stationary momentum equation and the heat equation and, second, the interface movement which (after transforming to a reference domain) leads to additional terms as well as time-dependent and non-periodic coefficients functions. 

As typical for a corrector estimate result, the main goal is to show that there is a constant $C>0$ which is independent on the particular choice of $\e$ such that
%
\begin{multline}\label{goal_est}
\|\Theta_\err^\e\|_{L^\infty(S\times\Omega)}+\|U_{\err}^\e\|_{L^\infty(S;L^2(\Omega))^{3}}+\|\nabla\Theta_\cor^\e\|_{L^2(S\times\Omega_A^\e)^3}+\|\nabla U_{\cor}^\e\|_{L^\infty(S;L^2(\Omega_A^\e))^{3\times3}}\\
+\e\|\nabla\Theta_\err\|_{L^2(S\times\Omega_B^\e)^3}+\e\left\|\nabla U_\cor^\e\right\|_{L^\infty(S;L^2(\Omega_B^\e))^{3\times3}}\leq C (\sqrt{\e}+\e).
\end{multline}
Here, $\Omega$ represents the full medium, $\Omega_A^\e$ the fast-heat-conducting connected matrix and $\Omega_B^\e$ the slow-heat-conducting inclusions.
For the definitions of the (error and corrector) functions, we refer the reader to the beginning of Section~\ref{sec:corrector}.

Unfortunately, in the general setting of a fully-coupled thermoelasticity problem with moving interface, such corrector estimates as stated in~\eqref{goal_est} seem to not be obtainable; in Section~\ref{subsec:overall}, we point out where and why the usual strategy for establishing such estimates is bound to fail.

Instead, we show that there are a couple of possible simplifications of the full model in which ~\eqref{goal_est} holds:
\begin{itemize}
\item[(a)] \emph{Weakly coupled problem}: If we assume either the \emph{mechanical dissipation} or the \emph{thermal stresses} to be negligible, we are lead to weakly coupled problems, where the desired estimates can be established successively, see Theorem~\ref{cor1}.
We note that the regularity requirements are higher in the case of no thermal stress compared to the case of mechanical dissipation.
\item[(b)] \emph{Microscale coupling}: If \emph{mechanical dissipation} and \emph{thermal stress} are only really significant in the slow-conducting component and negligible in the connected matrix part, the estimate hold, see Theorem~\ref{cor2}.
\end{itemize}
As pointed out in~\cite{W99}, neglecting the effect of mechanical dissipation is a step that is quite usual in modeling thermoelasticity problems.

Convergence rates for specific one-phase problems with periodic constants (some of them posed in perforated domains) were investigated in, e.g.,~\cite{BP13,BPC98,CP98}.
In~\cite{E04}, convergence rates for a complex nonlinear problem modeling liquid-solid phase transitions via a phase-field approach were derived.
A homogenization result including corrector estimates for a two-scale diffusion problem posed in a locally-periodic geometry was proven in~\cite{MV13,VM11}.
Here, similar to our scenario, the microstructures are non-uniform, non-periodic, and assumed to be \emph{a priori} known; the microstructures are however time independent and there are no coupling effects.
For some corrector estimate results in the context of thermo and elasticity problems, we refer to~\cite{BP13,STV13}, e.g.
We also want to point out to the newer and different philosophy in which the solutions are compared in the two-scale spaces (e.g., $L^2(\Omega;H_\#^1(Y_B)$) as opposed to the, possibly $\e$-dependent, spaces for the $\e$-problem (e.g., $L^2(\Omega_B^\e)$),\footnote{Here, we have used notation and domains as introduced in Section~\ref{sec:setting}.} a method which requires considerably less regularity on part of the solutions of the homogenized problems, we refer to~\cite{G04,MR16,R15}.

The paper is organized as follows: In Section~\ref{sec:setting}, we introduce the $\e$-microscopic geometry and formulate the thermoelasticity micro-problem in the moving geometry and transformed to a fixed reference domain and also state the homogenized two-scale problem.
The assumptions on our data, some regularity statements, and auxiliary estimate results are then collected in Section~\ref{sec:preliminaries}.
Finally, in Section~\ref{sec:corrector}, we focus on establishing convenient $\e$-control for the terms arising in the \emph{error formulation} (see equation~\eqref{corrector_eq}).
Based on these estimates, the corrector estimate~\eqref{goal_est} is then shown to hold for the above described cases $(a)$ and $(b)$.

\section{Setting}\label{sec:setting}
\subsection{Interface movement}
The following notation is taken from~\cite{EM17}.
Let $S=(0,T)$, $T>0$, be a time interval.
Let $\Omega$ be the interior of a union of a finite number of closed cubes $Q_j$, $1\leq j\leq n$, $n\in\N$, whose vertices are in $\Q^3$ such that, in addition, $\Omega$ is a Lipschitz domain.

In addition, we denote the outer normal vector of $\Omega$ with $\nu=\nu(x)$.
Let $Y=(0,1)^3$ be the open unit cell in $\R^3$.
Take $Y_A,$ $Y_B\subset Y$  two disjoint open sets, such that $Y_A$ is connected, such that $\Gamma:=\overline{Y_A}\cap\overline{Y_B}$ is a $C^3$ interface, $\Gamma=\partial Y_B$, $\overline{Y_B}\subset Y$, and $Y=Y_A\cup Y_B\cup \Gamma$.
With $n=n(y)$, $y\in\Gamma$, we denote the normal vector of $\Gamma$ pointing outwards of $Y_B$.

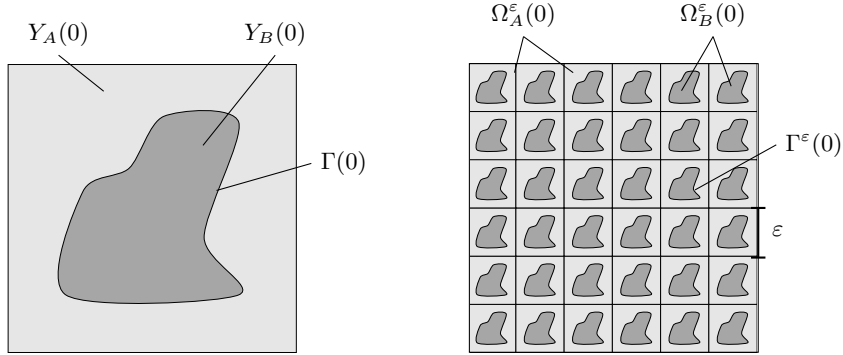
\begin{figure}[!ht]
\hspace{1cm}
\begin{tikzpicture}[scale=1.2]
	\pgftext{\includegraphics[width=0.215\textwidth,]{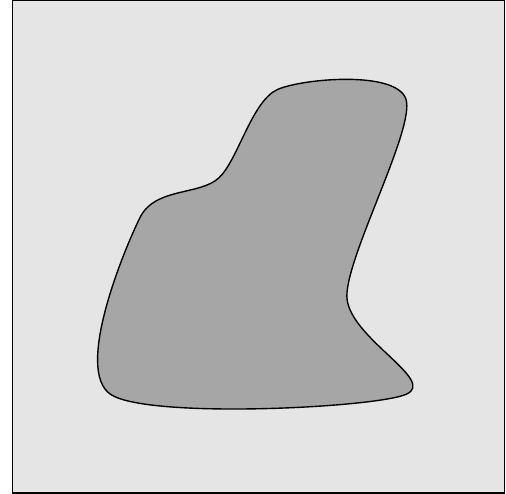}};
	\draw[thin] (0.6,0.7) -- (1.4,1.7);
	\draw (1.4,1.7) node[above] {{\small{$Y_B(0)$}}};
	\draw[thin] (-0.5,1.3) -- (-1,1.7);
	\draw (-1,1.7) node[above] {{\small{$Y_A(0)$}}};
	\draw[thin] (0.75,0.2) -- (1.8,0.5);
	\draw (1.8,0.5) node[right] {{\small{$\Gamma(0)$}}};
\end{tikzpicture}
\hspace{1cm}
\begin{tikzpicture}[scale=1.2]
	\pgftext{\includegraphics[width=0.208\textwidth]{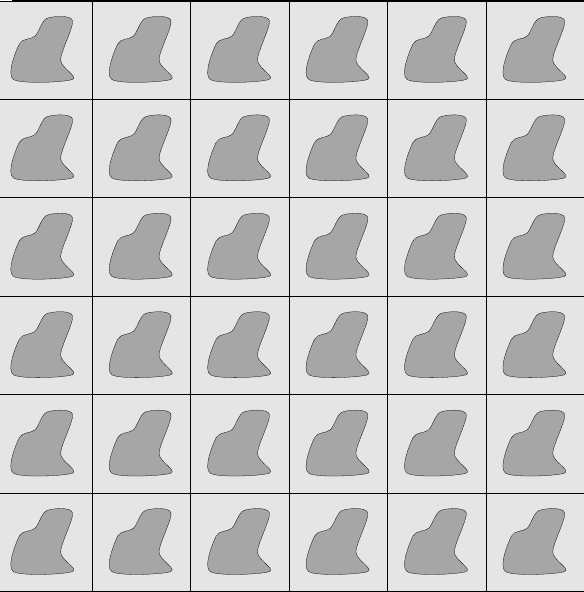}};
	\draw[thin] (-1.6,-1.6) rectangle (1.6,1.6);
	\draw[thin] (0.75,1.3) -- (1.1,1.9);
	\draw[thin] (1.3,1.35) -- (1.1,1.9);
	\draw (1.1,1.9) node[above] {{\small{$\Omega_B^\e(0)$}}};
	\draw[thin] (-0.45,1.5) -- (-1,1.9);
	\draw[thin] (-1.1,1.5) -- (-1,1.9);
	\draw (-1,1.9) node[above] {{\small{$\Omega_A^\e(0)$}}};
	\draw[thin] (0.9,0.2) -- (1.8,0.7);
	\draw (1.8,0.7) node[right] {{\small{$\Gamma^\e(0)$}}};
	\draw[thick] (1.52,0) -- (1.68,0);
	\draw[thick] (1.52,-0.55) -- (1.68,-0.55);
	\draw[thick] (1.6,0) -- (1.6,-0.55);
	\draw (1.65,-0.25) node[right] {{\small{$\e$}}};
\end{tikzpicture}
\caption{Reference geometry and the resulting $\e$-periodic initial configuration. Note that for $t\neq0$, these domains typically loose their periodicity.}
\label{s:fig}
\end{figure}

For $\e>0$, we introduce the $\e Y$-periodic, initial domains $\Omega_A^\e$ and $\Omega_B^\e$ and interface $\Gamma^\e$ representing the two phases and the phase boundary, respectively, via ($i\in\{A,B\}$)
\begin{align*}
	\Omega^\e_i=\Omega\cap\left(\bigcup_{k\in\Z^3}\e(Y_i+k)\right),\qquad
	\Gamma^\e=\Omega\cap\left(\bigcup_{k\in\Z^3}\e(\Gamma+k)\right).
\end{align*}
Here, for a set $M\subset\R^3$, $k\in\Z^3$, and $\e>0$, we employ the notation
$$
	\e(M+k):=\left\{x\in\R^3\ : \ \frac{x}{\e}-k\in M\right\}.
$$
From now on, we take $\e=(\e_n)_{n\in\N}$ to be a sequence of monotonically decreasing positive numbers converging to zero such that $\Omega$ can be represented as the union of cubes of size $\e$.
Note that this is possible due to the assumed structure of $\Omega$.

Here $n^\e=n(\frac{x}{\e})$, $x\in\Gamma^\e$, denotes the unit normal vector (extended by periodicity) pointing outwards $\Omega_B^\e$ into $\Omega_A^\e$.
The above construction ensures that  $\Omega_A^\e$ is connected and that $\Omega_B^\e$ is disconnected. 
We also have that $\partial\Omega_B^\e\cap\partial\Omega=\emptyset$.

We assume that $s\colon\overline{S}\times\overline\Omega\times\R^3\to\overline{Y}$ is a function such that
\begin{enumerate}
	\item $s\in C^2(\overline{S};C^2(\overline{\Omega})\times C^2_\#(Y))$,\footnote{%
	Here, and in the following, the $\#$ subscript denotes periodicity, i.e., for $k\in\N_0$, we have $C^k_\#(Y)=\{f\in C^k(\R^3): f(x+e_i)=f(x)\ \text{for all} \ x\in\R^3\}$, $e_i$ basis vector of $\R^3$.}
	\item $s(t,x,\cdot)_{|\overline{Y}}\colon\overline{Y}\to\overline{Y}$ is bijective for every $(t,x)\in\overline{S}\times\overline{\Omega}$,
	\item $s^{-1}\in C^2(\overline{S};C^2(\overline{\Omega})\times C^2_\#(Y))$,\footnote{Here, $s^{-1}\colon\overline{S}\times\overline\Omega\times\R^3\to\overline{Y}$ is the unique function such that $s(t,x,s^{-1}(t,x,y))=y$ for all $(t,x,y)\in\overline{S}\times\overline{\Omega}\in\overline{Y}$ extended by periodicity to all $y\in\R^3$.}
	\item $s(0,x,y)=y$ for all $x\in\overline{\Omega}$ and all $y\in\overline Y$,
	\item $s(t,x,y)=y$ for all $(t,x)\in\overline{S}\times\overline{\Omega}$ and all $y\in\partial Y$,
	\item there is a constant $c>0$ such that $\mathrm{dist}(\partial Y,\gamma)>c$ for all $\gamma\in s(t,x,\Gamma)$ and $(t,x)\in\overline{S}\times\overline{\Omega}$,
	\item $s(t,x,y)=y$ for all $(t,x)\in\overline{S}\times\overline{\Omega}$ and for all $y\in Y$ such that $\mathrm{dist}(\partial Y,y)<\frac{c}{2}$,
	\item there are constants $c_s,C_s>0$ such that
	      $$
					c_s\leq\det(\nabla s(t,x,y))\leq C_s,\quad (t,x,y)\in\overline{S}\times\overline\Omega\times\R^3
				$$
\end{enumerate}
and set the $(t,x)$-parametrized sets
\begin{align*}
	Y_A(t,x)=s(t,x,Y_A),\quad
	Y_B(t,x)=s(t,x,Y_B),\quad
	\Gamma(t,x)=s(t,x,\Gamma).
\end{align*}

We introduce the operations
\begin{align*}
	[\cdot]&\colon\R^3\to\Z^3,\quad [x]=k \ \text{such that}\  x-[x]\in Y,\\
	\{\cdot\}&\colon\R^3\to Y,\quad \{x\}=x-[x]
\end{align*}
and define the $\e$-dependent function\footnote{This is the typical notation in the context of homogenization via the \emph{periodic unfolding method}, see, e.g., \cite{CDG08, D12}.}
$$
	s^\e\colon\overline{S}\times\overline{\Omega}\to\R^3,\quad s^\e(t,x):=\e\left[\frac{x}{\e}\right]+\e s\left(t,\e\left[\frac{x}{\e}\right],\frac{x}{\e}\right).
$$
For $i\in\{A,B\}$ and $t\in\overline{S}$, we set the time dependent sets $\Omega_i^\e(t)$ and $\Gamma^\e(t)$ and the corresponding non-cylindrical space-time domains $Q_i^\e$ and space-time phase boundary $\Sigma^\e$ via
\begin{alignat*}{2}
	\Omega_i^\e(t)&=s^\e(t,\Omega_i^\e),&\qquad Q_i^\e&=\bigcup_{t\in S}\left(\{t\}\times\Omega_i^\e(t)\right),\\
	\Gamma^\e(t)&=s^\e(t,\Gamma^\e),&\qquad\Sigma^\e&=\bigcup_{t\in S}\left(\{t\}\times\Gamma^\e(t)\right),
\end{alignat*}
and denote by $n^\e=n^\e(t,x)$, $t\in S$, $x\in\Gamma^\e(t)$, the unit normal vector pointing outwards $\Omega_B^\e(t)$ into $\Omega_A^\e(t)$.
The time-dependent domains $\Omega_i^\e$ host the phases at time $t\in\overline{S}$ and model the movement of the interface $\Gamma^\e$.
We emphasize that, for any $t\neq0$, the sets $\Omega_A^\e(t)$, $\Omega_B^\e(t)$, and $\Gamma^\e(t)$ are, in general, not periodic.

We introduce the transformation-related functions (here, $W_\Gamma^\e$ is the normal velocity and $H_\Gamma^\e$ the mean curvature of the interface) via
\begin{alignat*}{2}
	F^\e&\colon\overline{S}\times\overline{\Omega}\to\R^{3\times3},\qquad &
		F^\e(t,x)&:=\nabla s^\e(t,x),\\
	J^\e&\colon\overline{S}\times\overline{\Omega}\to\R,\qquad&
		J^\e(t,x)&:=\det\left(\nabla s^\e(t,x)\right),\\
	v^\e&\colon\overline{S}\times\overline{\Omega}\to\R^{3},\qquad&
		v^\e(t,x)&:=\partial_ts^\e(t,x),\\
	W_\Gamma^\e&\colon\overline{S}\times\Gamma^\e\to\R,\quad&
		W_\Gamma^\e(t,x)&:=v^\e(t,x)\cdot n^\e(t,s^\e(t,x)),\\
	H_\Gamma^\e&\colon\overline{S}\times\Gamma^\e\to\R,\quad&
		H_\Gamma^\e(t,x)&:=-\dive\left((F^{\e})^{-1}(t,x)n^\e(t,s^\e(t,x))\right)
\end{alignat*}
for which we have the following estimates
\begin{multline}\label{s:estimate_movement}
	\left\|F^\e\right\|_{L^\infty(S\times\Omega)^{3\times3}}+\left\|(F^\e)^{-1}\right\|_{L^\infty(S\times\Omega)^{3\times3}}+\left\|J^\e\right\|_{L^\infty(S\times\Omega)}\\
	+\e^{-1}\left\|v^\e\right\|_{L^\infty(S\times\Omega)^3}
	+\e^{-1}\left\|W_\Gamma^\e\right\|_{L^\infty(S\times\Gamma^\e)}+\e\left\|H_\Gamma^\e\right\|_{L^\infty(S\times\Gamma^\e)}\leq C.
\end{multline}
By design, the constant $C$ entering~\eqref{s:estimate_movement} is independent on the choice of $\e$.
Note that the same estimates also hold for the time derivatives of these functions.
\subsection{$\e$-problem and homogenization result}

The bulk equations of the coupled thermoelasticity problem are given as (we refer to~\cite{B56,EM17,K79})
\begin{subequations}\label{p:full_problem_moving}
\begin{alignat}{2}
	-\dive(\CC_A^\e e(u_A^\e)-\alpha_A^\e\theta_A\mathds{I}_3)&=f_{u_A}^\e&\quad&\text{in}\ \ Q_A^\e,\label{p:full_problem_moving:1}\\
	-\dive(\CC_B^\e e(u_B^\e)-\alpha_B^\e\theta_B\mathds{I}_3)&=f_{u_B}^\e&\quad&\text{in}\ \ Q_B^\e,\label{p:full_problem_moving:2}\\
	\partial_t\left(\rho_Ac_{dA}\theta_A^\e+\gamma_A^\e\dive u_A^\e\right)-\dive(K_A^\e\nabla\theta_A^\e)&=f_{\theta_A}^\e&\quad&\text{in}\ \ Q_A^\e,\label{p:full_problem_moving:3}\\
	\partial_t\left(\rho_Bc_{dB}\theta_B^\e+\gamma_B^\e\dive u_B^\e\right)-\dive(K_B^\e\nabla\theta_B^\e)&=f_{\theta_B}^\e&\quad&\text{in}\ \ Q_B^\e.\label{p:full_problem_moving:4}
\end{alignat}
Here, $\CC_i^\e\in\R^{3\times3\times3\times3}$ are the \emph{stiffness} tensors, $\alpha_i^\e>0$ the \emph{thermal expansion} coefficients, $\rho_i>0$ the \emph{mass densities}, $c_{di}>0$ the \emph{heat capacities}, $\gamma_i^\e>0$ are the \emph{dissipation coefficients}, $K_i^\e\in\R^{3\times3}$ the \emph{thermal conductivities}, and $f_{u_i}^\e$, $f_{\theta_i}^\e$ are volume densities.
In addition, $e(v)=1/2(\nabla v+\nabla v^T)$ denotes the linearized strain tensor and $\mathds{I}_3$ the identity matrix.

At the interface between the phases, we assume continuite of both the temperature and deformation and the fluxes of force and heat densities to be given via the mean curvature and the interface velocity, resp.:\footnote{Here, the scaling via $\e^2$ counters the effects of both the interface surface area, note that $\e|\Gamma^\e|\in\mathcal{O}(1)$, and the curvature itself, note that $\e|\widetilde{H_\Gamma^\e}|\in\mathcal{O}(1)$.}
\begin{alignat}{2}
	\jump{u^\e}=0,\quad \jump{\theta^\e}&=0&\quad&\text{on}\ \ \Sigma^\e,\label{p:full_problem_moving:5}\\
	\jump{\CC^\e\e(u^\e)-\alpha^\e\theta^\e\mathds{I}_3}n^\e&=-\e^2H_\Gamma^\e n^\e&\quad&\text{on}\ \ \Sigma^\e,\\
	\jump{\rho c_d}\theta^\e W_{\Gamma}^\e+\jump{\gamma^\e\dive u^\e}W_\Gamma^\e-\jump{K^\e\nabla\theta^\e}\cdot n^\e&= L_{AB}W_\Gamma^\e&\quad&\text{in}\ \ \Sigma^\e.
\end{alignat}
Here, $\jump{v}:=v_A-v_B$ denotes the jump across the boundary separating phase $A$ from phase $B$, $\sigma_0>0$ is the coefficient of surface tension, and $L_{AB}\in\R$ is the latent heat

Finally, at the boundary of $\Omega$ and for the initial condition, we pose
\begin{alignat}{2}
	u_A^\e&=0&\quad&\text{on}\ \ S\times\partial\Omega_A^\e,\\
	\theta_A^\e&=0&\quad&\text{on}\ \ S\times\partial\Omega_A^\e,\\
	\theta^\e(0)&=\theta_{0}^\e&\quad&\text{on}\ \ \Omega,\label{p:full_problem_moving:9}\\
	u^\e(0)&=0&\quad&\text{on}\ \ \Omega,
\end{alignat}
where $\theta_{0}^\e$ is some (possibly highly heterogeneous) initial temperature distribution.
The scaling of the coefficients is chosen as
\begin{alignat*}{4}
	\CC_A^\e&=\CC_A,&\quad K_A^\e&=K_A,&\quad\alpha_A^\e&=\alpha_A,&\quad\gamma_A^\e&=\gamma_A,\\
	\CC_B^\e&=\e^2\CC_B,&\quad K_B^\e&=\e^2K_B,&\quad\alpha_B^\e&=\e\alpha_B,&\quad\gamma_B^\e&=\e\gamma_B.
\end{alignat*}
\end{subequations}
\begin{remark}
The simplified models described in the introduction (for which corrector estimates can be established) correspond to $\alpha_A=\alpha_B=0$ or $\gamma_A=\gamma_B=0$ (case (1), \emph{weakly coupled problem}) and $\alpha_A=\gamma_A=0$ (case (2), \emph{mirco coupled} problem).
\end{remark}

Now, taking the back-transformed quantities (defined on the initial, periodic domains $\Omega_i^\e$) $U_i^\e\colon S\times\Omega_i^\e\to\R^3$ and $\Theta_i^\e\colon S\times\Omega_i^\e\to\R^3$ given via $U_i^\e(t,x)=u_i^\e(t,s^{-1,\e}(t,x))$ and $\Theta_i^\e(t,x)=\theta_i^\e(t,s^{-1,\e}(t,x))$,\footnote{%
Here, $s^{-1,\e}\colon \overline{S}\times\overline{\Omega}\to\overline{\Omega}$ is the inverse function of $s^\e$.} we get the following problem in fixed coordinates (for more details regarding the transformation to a fixed domain, we refer to~\cite{D12,M08, PSZ13}):\footnote{Here, the superscript $r,\e$ denotes the transformed quantities (w.r.t~$s^\e$), for example $K_A^{r,\e}=J^\e(F^\e)^{-1}K_A(F^\e)^{-T}$ (cf.~\cite{EM17}).}
\begin{subequations}\label{p:ref}
\begin{alignat}{2} 
	-\dive\left(\CC_A^{r,\e}(U_A^\e)-\Theta_A^\e\alpha_A^{r,\e}\right)&=f_{u_A}^{r,\e} \quad &&\text{in}\ \ S\times\Omega_A^\e,\\
	-\dive\left(\e^2\CC_B^{r,\e} e(U_B^\e)-\e\Theta_B^\e\alpha_B^{r,\e}\right)&=f_{u_B}^{r,\e} \quad &&\text{in}\ \ S\times\Omega_B^\e,\\
	\begin{split}
	\partial_t\left(c_A^{r,\e}\Theta_A^\e+\gamma_A^{r,\e}:\nabla U_A^\e\right)-\dive\left(K_A^{r,\e}\nabla\Theta_A^\e\right)\quad\\
	-\dive\left(\left(c_A^{r,\e}\Theta_A^\e+\gamma_A^{r,\e}:\nabla U_A^\e\right)v^{r,\e}\right)&=f_{\theta_A}^{r,\e}
	\end{split}\quad&&\text{in}\ \ S\times\Omega_A^\e,\\
	\begin{split}
	\partial_t\left(c_B^{r,\e}\Theta_B^\e+\e\gamma_B^{r,\e}:\nabla U_B^\e\right)-\dive\left(\e^2K_B^{r,\e}\nabla\Theta_B^\e\right)\quad\\
	-\dive\left(\left(c_B^{r,\e}\Theta_B^\e+\e\gamma_B^{r,\e}:\nabla U_B^\e\right)v^{r,\e}\right)
	&=f_{\theta_B}^{r,\e}
	\end{split}\quad&&\text{in}\ \ S\times\Omega_B^\e,
\end{alignat}
\end{subequations}
complemented with interface transmission, boundary, and initial conditions.

Now, for $j,k\in\{1,2,3\}$ and $y\in Y$, set $d_{jk}=(y_j\delta_{1k},y_j\delta_{2k},y_j\delta_{3k})^T$.
For $t\in S$, $x\in\Omega$, let $\tau_j^\theta(t,x,\cdot)\in H^1_\per(Y_A)$, $\tau^u_{jk}(t,x,\cdot)$, $\tau^u(t,x,\cdot)\in H^1_\per(Y_A)^3$ be the solutions to the following variational cell problems\footnote{Here, and in the following, the superscript $r$ denotes the transformed quantities (w.r.t~$s$), e.g., $K_A^{r}=J(F)^{-1}K_A(F)^{-T}$.}
\begin{subequations}\label{cell}
\begin{align}
	0&=\int_{Y_A}K_A^{r}\left(\nabla_y\tau_j^\theta+e_j\right)\cdot\nabla_yv\dy\quad\text{for all}\ \ v\in H^1_\#(Y_A),\label{cell:1}\\
	0&=\int_{Y_A}\CC_A^{r}e_y(\tau^u_{jk}+d_{jk}):e_y(v)\dy\quad\text{for all}\ \ v\in H^1_\#(Y_A)^3,\label{cell:2}\\
	0&=\int_{Y_A}\CC_A^{r}e_y(\tau^u):e_y(v_A)\dy-\int_{Y_A}\alpha_A^{r}:\nabla_yv\dy\quad\text{for all}\ \ v\in H^1_\#(Y_A)^3\label{cell:3}.
\end{align}
\end{subequations}
Using these functions, we introduce the fourth rank tensor $\CC$ and the matrix $K$ via
\begin{alignat*}{2}
\left(\CC\right)_{i_1i_2j_1j_2}&=\CC_A^{r}e_y\left(\tau_{i_1i_2}^u+d_{i_1i_2}\right):e_y\left(\tau_{j_1j_2}^u+d_{j_1j_2}\right),\\
\left(K\right)_{ij}&=K_A^{r}\left(\nabla_Y\tau_j^\theta+e_j\right)\cdot\left(\nabla_Y\tau_i^\theta+e_i\right).
\end{alignat*}
Furthermore, we define the following set of averaged coefficients
\begin{alignat*}{2}
C^h&=\int_{Y_A}C\di{y},&\quad
	\alpha_A^{h}&=\int_{Y_A}\left(\alpha_A^{r}C_A^re_y(\tau^{m})\right)\dy,\\
H_\Gamma^h&=\int_\Gamma H_\Gamma n\di{s},&\quad
	f_u^{h}&=\int_{Y_A}f_{u_A}^{r}\dy+\int_{Y_B}f_{u_B}^{r}\dy,\\
c^{h}&=\rho_Ac_{dA}\left|Y_A\right|+\int_{Y_A}\gamma_A^{r}:\nabla_y\tau^u\dy,&\quad
	W_{\Gamma}^{h}&=\int_{\Gamma}W_\Gamma^{r}\di{s},\\
K_A^{h}&=\int_{Y_A}K\di{y},&\quad
	f_\theta^{h}&=\int_{Y_A}f_{\theta_A}^{r}\dy+\int_{Y_B}f_{\theta_B}^{r}\dy,\\
A^{h}(\Theta_B,U_B)&=\int_{Y_B}\left(c_B^{r}\Theta_B+\gamma_B^{r}U_B\right)\di{y},&\quad
\gamma_A^{h}&=\int_{Y_A}\left(\gamma_A^{r}+\gamma_A^{r}\nabla_y\tau_{jk}^{m}\right)\dy.
\end{alignat*}
After a homogenization procedure (the details of which are presented in~\cite{EM17}), we get the following upscaled two-scale model 

\begin{subequations}\label{p:refhom}
\begin{alignat}{2}
-\dive\left(\CC_A^{h}e(u_A)-\alpha_A^{h}\theta_A\right)&=f_u^{h}+H_\Gamma^{h}\quad &&\text{in}\ \ S\times\Omega,\label{h:hom_prob:1}\\
\partial_t\left(c^{h}\theta_A+\gamma_A^{h}:\nabla u_A+A^{h}(\Theta_B,U_B)\right)
-\dive\left(K_A^{h}\nabla\theta_A\right)&=f_\theta^{h}-W_{\Gamma}^{h}\quad &&\text{in}\ \ S\times\Omega\label{h:hom_prob:2},\\
-\dive_y\left(\CC_B^r e_y(U_B)-\alpha_B^r\Theta_B\right)&=f_{u_B}^r\ \ &&\text{in}\ \ S\times\Omega\times Y_B\label{h:hom_prob:3},\\
\begin{split}
\partial_t\left(c_B^r\Theta_B+\gamma_B^r:\nabla_y U_B\right)-\dive_y\left(K_B^r\nabla_Y\Theta_B\right)\qquad\\-\dive_y\left(\left(c_B^r\Theta_B^r+\gamma_B^r:\nabla_y U_B\right)v^r\right)&=f_{\theta_B}^r\end{split}\ \ &&\text{in}\ \ S\times\Omega\times Y_B,\label{h:hom_prob:4}\\
U_B=u_A,\ \ \Theta_B&=\theta_A\quad &&\text{on}\ \ S\times\Omega\times\Gamma,\label{h:hom_prob:5}
\end{alignat}
\end{subequations}
again, complemented with corresponding initial and boundary values.

\section{Preliminaries}\label{sec:preliminaries}
In this section, we will lay the groundwork for the corrector estimations that are done in Section~\ref{sec:corrector} in stating the existence and regularity results for the solutions and also giving some auxiliary estimates.

We introduce the spaces
$$
H^1(\Omega_A^\e;\partial\Omega)=\left\{u\in H^1(\Omega_A^\e) \ : \ u=0 \ \text{on} \ \partial\Omega\right\}
$$
and, for a Banach space $X$,
$$
\mathcal{W}(S;X)=\left\{u\in L^2(S;X) \ : \ \text{such that} \ \partial_tu\in L^2(S;X')\right\}.
$$
In general, we will not differentiate (in the notation) between a function defined on $\Omega$ and its restriction to $\Omega_A^\e$ or $\Omega_B^\e$ or between a function defined on one of those subdomains and its trivial extension to the whole of $\Omega$.
Here, and in the following, $C$, $C_1$, $C_2$ denote generic constants which are independent of $\e$ but whose values might change even from line to line.

For a function $f=f(x,y)$, we introduce the so called macroscopic reconstruction $\recon{f}=\recon{f}(x)=f(x,x/\e)$.
Note that, for general $f\in L^\infty(\Omega;H^1(Y)$, $\recon{f}$ may not even be measurable (see~\cite{Al92}); continuity in one variable, e.g., $f\in L^\infty(\Omega;C_\#(Y))$ is sufficient, though.
Applying the chain rule leads to $D\recon{f}=\recon{D_xf}+1/\e\recon{D_yf}$, where $D=\nabla, e(\cdot),\dive(\cdot)$, for sufficiently smooth functions $f$.

\paragraph{\textbf{Assumptions on the data (A1)}}
We assume that $\theta_0^\e\in H^1(\Omega)$, $f_{u_i}^{r,\e}\in C^1(S;L^2(\Omega_i^\e))$, $f_{\theta_i}^{r,\e}\in L^2(S\times\Omega_i^\e)$.
Furthermore, let $\theta_{A_0}\in H^1(\Omega)$, $f_{u_A}^h\in C^1(S;L^2(\Omega))$, $f_{\theta_A}^h\in L^2(S\times\Omega)$ for the macroscopic homogenized part and $\theta_{B_0}\in C(\Omega;H^1(Y_B))$, $f_{u_B}^r\in C^1(S;L^2(\Omega))$, $f_{\theta_B}^r\in L^2(S\times\Omega)$ for the two-scale part.
Moreover, we expect the following convergence rates to hold for our data
\begin{align*}
\left\|\mathds{1}_{\Omega_A^\e}\theta_0^\e-\theta_{A_0}\right\|_{L^2(\Omega_A^\e)}&\leq C\sqrt{\e},\\
\left\|\mathds{1}_{\Omega_B^\e}\theta_0^\e-\recon{\theta_{B_0}}\right\|_{L^2(\Omega_B^\e)}+\left\|f_{u_B}^{r,\e}-\recon{f_{u_B}^r}\right\|_{L^2(\Omega_B^\e)^3}+\left\|f_{\theta_B}^{r,\e}-\recon{f_{\theta_B}^r}\right\|_{L^2(\Omega_B^\e)}&\leq C\e
\end{align*}
and, in addition, to have the following estimates
\begin{align*}
\int_{\Omega_A^\e}\left|(f_{u_A}^{r,\e}-f_{u_A}^h)\p(x)\right|\di{x}&\leq C\e\|\p\|_{H^1(\Omega_A^\e)}\quad\text{for all} \ \p\in H^1(\Omega_A^\e;\partial\Omega)^3,\\
\int_{\Omega_A^\e}\left|(f_{\theta_A}^{r,\e}-f_{\theta_A}^h)\p(x)\right|\di{x}&\leq C\e\|\p\|_{H^1(\Omega_A^\e)}\quad\text{for all} \ \p\in H^1(\Omega_A^\e;\partial\Omega).
\end{align*}

If we are also interested in developing estimates for the time derivatives, we need stronger regularity assumptions.
\paragraph{\textbf{Assumptions on the data (A2)}}
Additionally to Assumptions (A1), we also expect the following convergence rates to hold:
\begin{align*}
\left\|\mathds{1}_{\Omega_A^\e}\theta_0^\e-\theta_{A_0}\right\|_{H^1(\Omega_A^\e)}&\leq C\sqrt{\e},\\
\left\|\mathds{1}_{\Omega_B^\e}\theta_0^\e-\recon{\theta_{B_0}}\right\|_{H^1(\Omega_B^\e)}+\left\|\partial_t(f_{u_B}^{r,\e}-\recon{f_{u_B}^r})\right\|_{L^2(\Omega_B^\e)^3}+\left\|\partial_t(f_{\theta_B}^{r,\e}-\recon{f_{\theta_B}^r})\right\|_{L^2(\Omega_B^\e)}&\leq C\e.
\end{align*}
Moreover, we assume
\begin{align*}
\int_{\Omega_A^\e}\left|\partial_t(f_{u_A}^{r,\e}-f_{u_A}^h)\p(x)\right|\di{x}&\leq C\e\|\p\|_{H^1(\Omega_A^\e)}\quad\text{for all} \ \p\in H^1(\Omega_A^\e;\partial\Omega)^3,\\
\int_{\Omega_A^\e}\left|\partial_t(f_{\theta_A}^{r,\e}-f_{\theta_A}^h)\p(x)\right|\di{x}&\leq C\e\|\p\|_{H^1(\Omega_A^\e)}\quad\text{for all} \ \p\in H^1(\Omega_A^\e;\partial\Omega).
\end{align*}

\subsection{Regularity results}
To be able to justify the steps in the estimates in Section~\ref{sec:corrector}, some of the involved functions need to be of higher regularity than is guaranteed via the standard $H^1$-theory for elliptic/parabolic problems.
In the following lemmas, we will collect the appropriate regularity results.
\begin{lemma}[Regularity of Cell Problem Solutions]\label{reg_cell}
The solutions of the problems \eqref{cell:1}-\eqref{cell:3} possess the regularity $\tau_j^\theta\in C^2(\overline{S};C^1(\overline{\Omega};W^{2,p}(Y_A)))$ and $\tau^u_{jk}$, $\tau^u\in C^2(\overline{S};C^1(\overline{\Omega};W^{2,p}(Y_A)^3))$ for some $p>3$.
\end{lemma}
\begin{proof}
The regularity w.r.t~$y\in Y_A$ can be derived from standard elliptic regularity theory (we refer to~\cite{GT13} for the general results and~\cite{E04} for the application to our case of the cell problems).
The rest is a direct consequence of the regularity (w.r.t.~$t\in S$ and $x\in\Omega$) of the involved coefficients.
\end{proof}
Note that this implies, in particular, that the cell problem functions and their gradients (w.r.t.~$y\in Y$) are bounded and that their macroscopic reconstructions are well-defined measurable functions.

In the following, we denote $U^\e=(U_A^\e,U_B^\e)$ and $\Theta^\e=(\Theta_A^\e,\Theta_B^\e)$.

\begin{lemma}[Existence and Regularity Theorem for the $\e$-Problem]\label{reg_epsilon}
There is a unique $(U^\e,\Theta^\e)\in \mathcal{W}(S;H_0^1(\Omega)^3\times H_0^1(\Omega))$ solving the variational system~\eqref{p:ref} for which standard energy estimates hold independently of the parameter $\e$.
Furthermore, this solution possesses the regularity $(U^\e,\Theta^\e)\in C^1(S;H^2(\Omega_A^\e)^3\times H^2(\Omega_B^\e)^3)\times L^2(S;H^2(\Omega_A^\e)\times H^2(\Omega_B^\e))$ such that $\partial_t\Theta^\e\in L^2(S;H^1(\Omega))$.
\end{lemma}
\begin{proof}
The proof of the existence of a unique solution and of the energy estimates is given in~\cite[Theorem 3.7, Theorem 3.8]{EM17}.
As a linear transmission problem (with sufficiently regular coefficients), regularity results apply (we refer to, e.g.,~\cite{E10}; see, also,~\cite{SM02} for a similar coupling problem).
\end{proof}

Since $s^\e$ is a diffeomorphism, this leads to a unique solution to the moving interface problem, also.
However, while the solution has $H^2$-regularity, its second derivatives are not necessarily bounded (and in general will not be) independently of $\e>0$.

\begin{lemma}[Existence and Regularity Theorem for the Homogenized Problem]\label{reg_homo}
There is a unique $$(u_A,\theta_A,U_B,\Theta_B)\in\mathcal{W}(S;H_0^1(\Omega)^3\times H_0^1(\Omega)\times L^2(\Omega;W^{1,2}(Y_B)^3)\times L^2(\Omega;W^{1,2}(Y_B)))$$ solving the variational system~\eqref{p:refhom}.
Furthermore, $(u_A,\theta_A)\in C^1(S;H^2(\Omega)^3)\times L^2(S;H^2(\Omega))$ such that $\partial_t\theta_A\in L^2(S;H^1(\Omega))$ and
$(U_B,\Theta_B)\in C^1(S;H^2(\Omega;H^2(Y_B)^3))\times L^2(S;H^2(\Omega;H^2(Y_B)))$ such that $\partial_t\Theta_B\in L^2(S;H^2(\Omega;H^1(Y_B)))$.
\end{lemma}
\begin{proof}
Existence of a solution is given via the two-scale homogenization procedure outlined in~\cite{EM17} and uniqueness for this linear coupled transmission problem can then be shown using energy estimates.
As to the higher regularity, this follows, again, via the regularity of domain, coefficients, and data, we refer to results outlined in~\cite{E10, SM02, VM11}.
\end{proof}
\subsection{Auxiliary estimates}

For the transformation related quantities, we have the following estimates available as stated in Lemma~\ref{lemma:estimates_trafo} and Lemma~\ref{lemma:average}.
\begin{lemma}\label{lemma:estimates_trafo}
There is a constant $C>0$ independent of the parameter $\e$ such that
\begin{multline*}
\|F^\e-\recon{F}\|_{L^\infty(S\times\Omega)^{3\times3}}+\|J^\e-\recon{J}\|_{L^\infty(S\times\Omega)}
+\|v^\e-\e\recon{v}\|_{L^\infty(S\times\Omega)}\\ +\|W_\Gamma^\e-\e\recon{W_\Gamma}\|_{L^\infty(S\times\Gamma)}+\|H_\Gamma^\e-\recon{H_\Gamma}\|_{L^\infty(S\times\Gamma)}
\leq C\e
\end{multline*}
The same estimates hold for the time derivatives of those functions.
\end{lemma}
\begin{proof}
We show this, by way of example, only for $F^\e$, the other estimates follow in the same way:
\begin{align*}
\|F^\e-\recon{F}\|_{L^\infty(S\times\Omega)^{3\times3}}&=\esssup_{(t,x)\in S\times\Omega}\left|\nabla_y s\left(t,\e\left[\frac{x}{\e}\right],{\frac{x}{\e}}\right)-\nabla_y s\left(t,x,{\frac{x}{\e}}\right)\right\|\leq L\frac{\e}{\sqrt{2}}.
\end{align*}
Here, $L$ is the Lipschitz constant of $F^\e$ w.r.t.~$x\in\Omega$ (uniform in $S\times Y$).
\end{proof}
Based on these estimates and due to the fact that all material parameters are assumed to be constant in the moving geometry, we get the same estimates for the material parameters ($K^{r,e}$, $\alpha^{r,\e}$ and so on) in the reference configuration.

The following lemma is concerned with $\e$-independent estimates for the macroscopic reconstruction of periodic functions with zero average.
There are several different but similar theorems that can be found in the literature regarding corrector estimates in the context of homogenization, we refer to, e.g., \cite{CPS07, CP98, E04, MV13}, but for our purposes the following version suffices:
\begin{lemma}\label{lemma:average}
Let $f\in L^2(S\times\Omega_A^\e;C_\#(Y)$	such that
$$
\int_{Y_A}f(t,x,y)\di{y}=0\quad \text{a.e.~in}\ \ S\times\Omega_A^\e.
$$
Then, there is a constant $c>0$ such that, independently of $\e$,
$$
\int_{\Omega_A^\e}\left|\recon{f}(t,x)\p(x)\right|\di{x}\leq C\e\|\p\|_{H^1(\Omega_A^\e)}.
$$
for all $\p\in H^1(\Omega_A^\e;\partial\Omega)$.
\end{lemma}
\begin{proof}
This can be proven similarly to the corresponding statements in~\cite[Lemma 3]{CP98} and \cite[Lemma 5.2]{MV13}.
\end{proof}

\section{Corrector estimates}\label{sec:corrector}
In this section, we are concerned with the actual corrector estimates.
%
%
%
%
Reconstructing micro-solutions from the homogenized functions via the $\recon{\cdot}$-operation and subtracting the heat equations~\eqref{p:full_problem_moving:1},~\eqref{p:full_problem_moving:2},~\eqref{h:hom_prob:2}, and~\eqref{h:hom_prob:4} and momentum equations~\eqref{p:full_problem_moving:1},~\eqref{p:full_problem_moving:1},~\eqref{h:hom_prob:1}, and~\eqref{h:hom_prob:3} we get
\begin{subequations}\label{corrector_eq}
\begin{alignat}{2}
\partial_t\left(c_A^{r,\e}\Theta_A^\e-\kappa_Ac^{h}\theta_A\right)+\partial_t\left(\gamma_A^{r,\e}:\nabla U_A^\e-\kappa_A\gamma_A^{h}:\nabla u_A\right)\qquad&\notag\\
	-\dive\left(K_A^{r,\e}\nabla\Theta_A^\e-\kappa_AK_A^{h}\nabla\theta_A\right)-\kappa_AW_{\Gamma}^{h}&=f_{\theta_A}^{r,\e}-\kappa_Af_\theta^{h},\label{corrector_eq:ha}\\
	-\dive\left(\CC_A^{r,\e}(U_A^\e)-\kappa_A\CC_A^{h}e(u_A)-\alpha_A^{r,\e}\Theta_A^\e+\kappa_A\alpha_A^{h}\theta_A\right)+H_\Gamma^{h}&=f_{u_A}^{r,\e}-\kappa_Af_u^{h},\label{corrector_eq:ma}\\[0.5cm]
	\partial_t\left(c_B^{r,\e}\Theta_B^\e-\recon{c_B^{r}\Theta_B}\right)+\partial_t\left(\e\gamma_B^{r,\e}:\nabla U_B^\e-\recon{\gamma_B^{r}:\nabla_Y U_B}\right)\qquad\notag\\
	-\dive\left(\e^2K_B^{r,\e}\nabla\Theta_B^\e\right)+\recon{\dive_Y\left(K_B^{r}\nabla_Y\Theta_B\right)}&=f_{\theta_B}^{r,\e}-\left[f_{\theta_B}^{r}\right]_\e,\label{corrector_eq:hb}\\
	-\dive\left(\e^2\CC_B^{r,\e} e(U_B^\e)-\e\alpha_B^{r,\e}\Theta_B^\e\right)+\recon{\dive_Y\left(\CC_B^{r} e_y(U_B)-\alpha_B^{r}\Theta_B\mathds{I}_3\right)}&=f_{u_B}^{r,\e}-\recon{f_{u_B}^{r}}\label{corrector_eq:mb}.
\end{alignat}
\end{subequations}
These equations hold in $\Omega_A^\e$ and $\Omega_B^\e$, respectively.
Using the interface and boundary conditions for both the $\e$-problem and the homogenized problem and then doing an integration by parts, these equations correspond to a variational problem in $H^{-1}(\Omega)$.

Our strategy in establishing the estimates is as follows: After introducing error and corrector functions and doing some further preliminary estimates, we first, in Section~\ref{subsec:momentum}, concentrate on the momentum part, i.e., equations~\eqref{corrector_eq:ma} and~\eqref{corrector_eq:mb}.
Here, we take the different terms arising in the weak formulation and estimate them individually using the results from Section~\ref{sec:preliminaries} and usual energy estimation techniques.
Combining those estimates, it is shown that the \emph{mechanical error} can be controlled by the \emph{heat error}, see Remark~\ref{rem:mech}. 
Then, in Section~\ref{subsec:heat}, we basically do the same for the heat conduction part, i.e., equations~\eqref{corrector_eq:ha} and~\eqref{corrector_eq:hb}, thereby arriving at the corresponding result that the \emph{heat error} is controlled by the \emph{mechanical error} (and the time derivative of the \emph{mechanical error}), see Remark~\ref{rem:heat}.

Finally, in Section~\ref{subsec:overall}, we go about combining those individual estimates.
Here, we show that for the scenarios $(a)$ (Theorem~\ref{cor1}) and $(b)$ (Theorem~\ref{cor2}) (as described in the introduction), we get the desired estimates, i.e., equation~\eqref{goal_est}.
Moreover, we point out why the same strategy will not work for the full problem.

Now, we introduce the functions\footnote{The subscripts ``$\err$'' and ``$\cor$'' for \emph{error} and \emph{corrector}, respectively.}
\begin{alignat*}{2}
U_\err^\e&=\begin{cases}U_A^\e-u_A & \text{in}\ S\times\Omega_A^\e\\ U_B^\e-\recon{U_B} & \text{in}\ S\times\Omega_B^\e\end{cases},\quad&
\Theta_\err^\e&=\begin{cases}\Theta_A^\e-\theta_A & \text{in}\ S\times\Omega_A^\e\\ \Theta_B^\e-\recon{\Theta_B} & \text{in}\ S\times\Omega_B^\e\end{cases},\\
U_\cor^\e&=\begin{cases}U_\err^\e-\e\recon{\widetilde{U}}& \text{in}\ S\times\Omega_A^\e\\ U_\err^\e & \text{in}\ S\times\Omega_B^\e\end{cases},\quad&
\Theta_\cor^\e&=\begin{cases}\Theta_\err^\e-\e\recon{\widetilde{\Theta}} & \text{in}\ S\times\Omega_A^\e\\ \Theta_\err^\e & \text{in}\ S\times\Omega_B^\e\end{cases}.
\end{alignat*}
The functions $\widetilde{U}$ and $\widetilde{\theta}$ are exactly the functions arising in the two-scale limits of the gradients of $U_A^\e$ and $\Theta_A^\e$, respectively, and are given as (cf.~\cite{EM17})
\begin{align*}
\widetilde{U}=\sum_{j,k=1}^3\tau_{jk}^ue(u_A)_{jk}+\tau^u\theta_A,\qquad
\widetilde{\Theta}=\sum_{j=1}^3\tau^\theta\partial_j\theta_A.
\end{align*}
Due to the corrector part (namely, $\e\recon{\widetilde{\Theta}}$ and $\e\recon{\widetilde{U}}$, respectively), our corrector functions will, in general, not vanish at $\partial\Omega$ and are therefore not valid choices of test functions for a weak variational formulation of the system given via equations~\eqref{corrector_eq:ha}-\eqref{corrector_eq:mb}.
Because of that, we introduce a smooth cut-off function $m^\e\colon\Omega\to[0,1]$ such that $m^\e(x)=0$ for all $x\in\Omega_A^\e$ such that $\dist(x,\partial\Omega)\leq\e c/2$ and $m^\e(x)=1$ for all $x\in\Omega_A^\e$ such that $\dist(x,\partial\Omega)\geq\e c$. 
Furthermore, we require the estimate
\begin{align}\label{estimate_cut}
\sqrt{\e}\left\|\nabla m^\e\right\|_{L^2(\Omega)}+\sqrt{\e^3}\left\|\Delta m^\e\right\|_{L^2(\Omega)}+\frac{1}{\sqrt{\e}}\left\|1-m^\e\right\|_{L^2(\Omega)}\leq C
\end{align}
to hold independently of the parameter $\e$.
For this cut-off function $m^\e$, we set
\begin{align*}
U_{\cor0}^\e&=\begin{cases}U_\cor^\e+(1-m^\e)\e\recon{\widetilde{U}} & \text{in}\ S\times\Omega_A^\e\\ U_\cor^\e-\e\recon{\widetilde{U}} & \text{in}\ S\times\Omega_B^\e\end{cases},\quad&
\Theta_{\cor0}^\e&=\begin{cases}\Theta_\cor^\e+(1-m^\e)\e\recon{\widetilde{\Theta}}& \text{in}\ S\times\Omega_A^\e\\\Theta_\cor^\e-\e\recon{\widetilde{\Theta}} & \text{in}\ S\times\Omega_B^\e\end{cases}.
\end{align*}
Obviously, we then have $\Theta_{\cor0}^\e\in H_0^1(\Omega)$ and $U_{\cor0}^\e\in H_0^1(\Omega)^3$.
Owing to the regularity of $\widetilde{U}$ and $\widetilde{\Theta}$ (Lemma~\ref{reg_cell}) and the estimate~\eqref{estimate_cut} for $m^\e$, these modified correctors admit the following $\e$-independent estimates (for $i=A,B$)\footnote{The same estimates hold when replacing the linearized strain tensor with the gradient operator.}
\begin{align*}
\left\|U_{\err}^\e\right\|_{L^2(\Omega_i^\e)^3}&\leq\left\|U_{\cor0}^\e\right\|_{L^2(\Omega_i^\e)^3}+C\e\leq\left\|U_{\err}^\e\right\|_{L^2(\Omega_i^\e)^3}+2C\e,\\
\left\|e(U_{\cor}^\e)\right\|_{L^2(\Omega_A^\e)^{3\times3}}&\leq\left\|e( U_{\cor0}^\e)\right\|_{L^2(\Omega_A^\e)^{3\times3}}+C(\sqrt{\e}+\e)\leq\left\|e(U_{\cor}^\e)\right\|_{L^2(\Omega_A^\e)^{3\times3}}+2C(\sqrt{\e}+\e),\\
\e\left\|e( U_{\cor}^\e)\right\|_{L^2(\Omega_B^\e)^{3\times3}}&\leq\e\left\|e(U_{\cor0}^\e)\right\|_{L^2(\Omega_B^\e)^{3\times3}}+C(\sqrt{\e}+\e)\leq\e\left\|e(U_{\cor}^\e)\right\|_{L^2(\Omega_B^\e)^{3\times3}}+2C(\sqrt{\e}+\e),\\[0.3cm]
\left\|\Theta_{\err}^\e\right\|_{L^2(\Omega_i^\e)}&\leq\left\|\Theta_{\cor0}^\e\right\|_{L^2(\Omega_i^\e)}+C\e\leq\left\|\Theta_{\err}^\e\right\|_{L^2(\Omega_i^\e)}+2C\e,\\
\left\|\nabla \Theta_{\cor}^\e\right\|_{L^2(\Omega_A^\e)^{3}}&\leq\left\|\nabla \Theta_{\cor0}^\e\right\|_{L^2(\Omega_A^\e)^{3}}+C(\sqrt{\e}+\e)\leq\left\|\nabla \Theta_{\cor}^\e\right\|_{L^2(\Omega_A^\e)^{3}}+2C(\sqrt{\e}+\e),\\
\e\left\|\nabla \Theta_{\cor}^\e\right\|_{L^2(\Omega_B^\e)^{3}}&\leq\e\left\|\nabla \Theta_{\cor0}^\e\right\|_{L^2(\Omega_B^\e)^{3}}+C(\sqrt{\e}+\e)\leq\e\left\|\nabla \Theta_{\cor}^\e\right\|_{L^2(\Omega_B^\e)^{3}}+2C(\sqrt{\e}+\e).
\end{align*}
%
Applying Korns inequality to $U_{\cor0}$ (see~\cite[Lemma 3.1]{EM17}) and using the above estimates, we then get
\begin{multline}\label{estimate_ucor}
\left\|U_\err^\e\right\|_{L^2(\Omega)}+\left\|\nabla U_{\cor}^\e\right\|_{L^2(\Omega_A^\e)^3}+\e\left\|\nabla U_{\cor}^\e\right\|_{L^2(\Omega_B^\e)^{3\times3}}\\ \leq
\left\|e(U_{\cor}^\e)\right\|_{L^2(\Omega_A^\e)^3}+\e\left\|e(U_{\cor}^\e)\right\|_{L^2(\Omega_B^\e)^{3\times3}}+C(\sqrt{\e}+\e).
\end{multline}
\subsection{Estimates for the momentum equations}\label{subsec:momentum}
Let us first concentrate on the mechanical part of the corrector equations, namely equations~\eqref{corrector_eq:ma} and~\eqref{corrector_eq:mb}.
After integrating over $\Omega$, multiplying with a test function $\p\in H_0^1(\Omega)^3$, and integrating by parts, we get
\begin{multline*}
\underbrace{\int_{\Omega_A^\e}\left(\CC_A^{r,\e} e(U_A^\e)-\kappa_A\CC_A^{h}e(u_A)\right):e(\p)\di{x}-\int_{\Gamma^\e}\kappa_A\CC_A^{h}e(u_A)n^\e\cdot\p\di{s}}_{=:I_1^\e(t,\p)}\\
+\underbrace{\e^2\int_{\Omega_B^\e}\left(\CC_B^{r,\e} e(U_B^\e)-\recon{\CC_B^{r}}e(\recon{U_B})\right):e(\p)\di{x}}_{=:I_2^\e(t,\p)}\\
\underbrace{-\int_{\Omega_A^\e}\left(\alpha_A^{r,\e}\Theta_A^\e-\kappa_A\alpha_A^{h}\theta_A\right):\nabla\p\di{x}+\int_{\Gamma^\e}\kappa_A\alpha_A^{h}\theta_An^\e\cdot\p\di{s}}_{=:I_3^\e(t,\p)}\\\
+\underbrace{\e\int_{\Omega_B^\e}\left(\alpha_B^{r,\e}\Theta_B^\e-\recon{\alpha_B^{r}\Theta_B}\right):\nabla\p\di{x}}_{=:I_4^\e(t,\p)}\\
=\underbrace{\int_{\Omega_A^\e}\left(f_{u_A}^{r,\e}-\kappa_A\int_{Y_A}f_{u_A}^{r}\di{y}\right)\cdot\p\di{x}}_{=:I_5^\e(t,\p)}
+\underbrace{\e^2\int_{\Gamma^\e}\recon{\CC_B^r}e(\recon{U_B})n^\e\cdot\p\di{s}-\int_{\Omega_A^\e}\kappa_A\int_{Y_B}f_{u_B}^{r}\di{y}\cdot\p\di{x}}_{=:I_6^\e(t,\p)}\\
+\underbrace{\e^2\int_{\Gamma^\e}H_{\Gamma}^{r,\e}n^\e\cdot\p\di{s}-\int_{\Omega_A^\e}\kappa_AH_\Gamma^{h}\cdot\p\di{x}}_{=:I_7^\e(t,\p)}
+\underbrace{\int_{\Omega_B^\e}\left(f_{u_B}^{r,\e}-\recon{f_{u_B}^{r}}\right)\cdot\p\di{x}}_{=:I_8^\e(t,\p)}
+I_9^\e(t,\p),
\end{multline*}
where 
$$
I_9^\e(t,\p):=\int_{\Omega_B^\e}\left(\e^2\dive\left(\recon{\CC_B^{r}e_x(U_B)}\right)+\e\recon{\dive_x\left(\CC_B^{r}e_y(U_B)}+\e\recon{\dive_x\left(\alpha_B^{r}\Theta_B\right)}\right)\right)\cdot\p\di{x}.
$$
We now go on estimating these terms individually and then combine the resulting estimates. 
We proceed for both $U_{\cor0}^\e$ and $\partial_tU_{\cor0}$ as test function choices.
While the first choice is the natural one for energy estimates, the second choice is needed in order to merge those estimates with the heat equation estimates.

Taking a look at $I_1^\e$, we calculate
\begin{multline*}
	I_1^\e(t,\p)=\int_{\Omega_A^\e}\CC_A^{r,\e} e(U_{\cor}^\e):e(\p)\di{x}\\ +\int_{\Omega_A^\e}\left(\CC_A^{r,\e} e(U_A^\e-U_{\cor}^\e)-\kappa_A\CC_A^{h}e(u_A)\right):e(\p)\di{x}-\int_{\Gamma^\e}\kappa_A\CC_A^{h}e(u_A)n^\e\cdot\p\di{s}
\end{multline*}
We also see (via the definition of $\CC$ and $\widetilde{U}$)
\begin{multline*}
\CC_A^{r,\e}e(U_A^\e-U_{\cor}^\e)
=\recon{\CC}e(u_A)+\recon{\CC_A^{r}e_y(\tau^u)}\theta_A\\
+\underbrace{\left(\CC_A^{r,\e}-\recon{\CC_A^{r}}\right)\left(e(u_A)+\recon{e_y\left(\widetilde{U}\right)}\right)+\e\CC_A^{r,\e}\recon{e_x\left(\widetilde{U}\right)}}_{R_1^\e},
\end{multline*}
where, using Lemma~\ref{lemma:average}, the estimate
$$
\left|\int_{\Omega_A^\e}R_1^\e\p\dx\right|\leq C\e\|p\|_{H^1(\Omega_A^\e)}
$$
holds.
Now, seeing that $\int_\Gamma \CC e(u_A)n\ds=0$ and $\dive_y(\CC e(u_A))=0$ a.e.~in $S\times\Omega$ and using Lemma~\ref{lemma:average}, we get
\begin{multline*}
\left|\int_{\Omega_A^\e}\left(\recon{\CC}-\kappa_A(0)\CC_A^{h}\right)e(U_A):e(\p)\di{x}
	-\int_{\Gamma^\e}\kappa_A(0)\CC e(u_A) n^\e\cdot\p\di{x}\right|\\
	=\left|\int_{\Omega_A^\e}\dive\left(\left(\recon{\CC}-\kappa_A(0)\CC_A^{h}\right)e(U_A)\right)\cdot\p\di{x}\right|
	\leq C\e\|\p\|_{H^1(\Omega_A^\e)}
\end{multline*}
for all $\p\in H^1(\Omega_A^\e;\partial\Omega)$.
Also,
\begin{align*}
\int_{\Omega_A^\e}\CC_A^{r,\e} e(U_{\cor}^\e):e(U_{\cor0}^\e)\di{x}&\geq C_1\|e(U_{\cor}^\e)\|_{L^2(\Omega_A^\e)^{3\times3}}^2
+\e\int_{\Omega_A^\e}\CC_A^{r,\e}e(U_\cor^\e):e\left((1-m^\e)\recon{\widetilde{U}}\right)\di{x}\\
&\geq \frac{C_1}{2}\|e(U_{\cor}^\e)\|_{L^2(\Omega_A^\e)^{3\times3}}^2-C_2(\e+\e^2).
\end{align*}
Here, the second inequality can be shown using the assumptions on $m^\e$ and the known estimates of the involved functions.
In summary, for $I_1^\e$, we arrive at
\begin{equation}\label{estimate_i1}
I_1^\e(t,U_{\cor0}^\e)-\int_{\Omega_A^\e}\recon{\CC_A^{r}e_y(\tau^u)}\theta_A:e(U_{\cor0}^\e)\di{x}\geq C_1\|e(U_{\cor}^\e)\|^2_{L^2(\Omega_A^\e)^{3\times3}}-C_2\left(\e+\e^2\right).
\end{equation}
Now, going on with $I_2^\e$, it is easy to see that
\begin{align}\label{estimate_i2}
I_2^\e(t,U_{\cor0}^\e)&=\e^2\int_{\Omega_B^\e}\recon{\CC_B^{r}}\left(e(U_B^\e)-e(\recon{U_B})\right):e(U_\cor)\di{x}\notag\\
&\hspace{3cm}+\e^2\int_{\Omega_B^\e}\left(\CC_B^{r,\e}-\recon{\CC_B^{r}}\right)e(U_B^\e):e(U_\cor)\di{x}\notag\\
&\geq C_1\e^2\left\|e(U_\cor^\e)\right\|_{L^2(\Omega_A^\e)^{3\times3}}^2-C_2\e^2.
\end{align}
Going forward with term $I_3^\e$, we decompose
\begin{equation}\label{decompose}
\alpha_A^{r,\e}\Theta_A^\e-\kappa_A\alpha_A^{h}\theta_A=\left(\alpha_A^{r,\e}-\recon{\alpha_A^{r}}\right)\Theta_A^\e+\recon{\alpha_A^{r}}\Theta_\err^\e+\left(\recon{\alpha_A^{r}}-\kappa_A\alpha_A^{h}\right)\theta_A
\end{equation}
from which we can estimate the first two terms as
\begin{align*}
\left|\int_{\Omega_A^\e}\left(\left(\alpha_A^{r,\e}-\recon{\alpha_A^{r}}\right)\Theta_A^\e+\recon{\alpha_A^{r}}\Theta_\err^\e\right):\nabla\p\right|\leq C\left(\e+\|\Theta_\err^\e\|_{L^2(\Omega_A^\e)}\right)\|\nabla\p\|_{L^2(\Omega_A^\e)}.
\end{align*}
For the remaining term of equation~\eqref{decompose} combined with the interface integral part of $I_3^\e$, we get
\begin{multline*}
\int_{\Omega_A^\e}\left(\recon{\alpha_A^{r}}-\kappa_A\alpha_A^{h}\right)\theta_A:\nabla\p\di{x}+\int_{\Gamma^\e}\kappa_A\alpha_A^{h}\theta_An^\e\cdot\p\di{s}\\
=-\int_{\Omega_A^\e}\recon{\dive_x\left(\left(\alpha_A^{r}-\kappa_A\int_{Y_A}\alpha_A^{r}\di{y}\right)\theta_A\right)}\cdot\p\di{x}
-\int_{\Omega_A^\e}\kappa_A\dive\left(\int_{Y_A}C_A^re_y(\tau^u)\di{y}\,\theta_A\right)\cdot\p\di{x}\\
-\int_{\Gamma^\e}\recon{\alpha_A^{r}}\theta_An^\e\cdot\p\di{s}
-\frac{1}{\e}\int_{\Omega_A^\e}\recon{\dive_y\left(\alpha_A^{r}\theta_A\right)}\cdot\p\di{x}
\end{multline*}
We apply Lemma~\ref{lemma:average} to
\begin{align*}
f_1&=\dive_x\left(\left(\alpha_A^{r}-\kappa_A\int_{Y_A}\alpha_A^{r}\di{y}\right)\theta_A\right),\\
f_2&=\dive_x\left(\left(\CC_A^{r}e_y(\tau^u))-\kappa_A\int_{Y_A}C_A^re_y(\tau^u)\di{y}\right)\theta_A\right)
\end{align*}
and recall that $\tau^u$ is a solution of the cell problem~\eqref{cell:3} (because of this the ``$\frac{1}{\e}\recon{\dive_y}$''-terms vanish) which leads to
\begin{equation}\label{estimate_i3}
\left|I_3^\e(t,\p)+\int_{\Omega_A^\e}\recon{\CC_A^{r}e_y(\tau^u)}\theta_A:e(\p)\di{x}\right|\leq
C\left(\e+\|\Theta_\err^\e\|_{L^2(\Omega_A^\e)}\right)\|\p\|_{H^1(\Omega_A^\e)}.
\end{equation}
Similarly as with $I_2^\e$, for the thermo-elasticity term $I_4^\e$, we get
\begin{align}\label{estimate_i4}
\left|I_4^\e(t,\p)\right|\leq C\|\Theta_\err^\e\|^2_{L^2(\Omega_B^\e)}+\e^2\|\nabla\p\|^2_{L^2(\Omega_B^\e)^{3\times3}}.
\end{align}
Next, we take the mean curvature error term $I_7^\e$
\begin{align*}
\left|I_7^\e(t,\p)\right|\leq\e^2\left|\int_{\Gamma^\e}\left(H_{\Gamma}^{r,\e}-\recon{H_{\Gamma}^r}\right)n^\e\cdot\p\di{s}\right|+\left|\e^2\int_{\Gamma^\e}\recon{H_{\Gamma}^r}n^\e\cdot\p\di{s}-\int_{\Omega_A^\e}\kappa_AH_\Gamma^{h}\cdot\p\di{x}\right|.
\end{align*}
Now, in view of the assumptions on our data and for the source density errors (stated in Assumptions (A1)) and the curvature estimate, using Lemma~\ref{lemma:estimates_trafo} (for the functional $I_6^\e(t,\p)$), and the boundedness of the functions involved in $I_9^\e$, we estimate
\begin{align}\label{estimate_i5}
\left|I_5^\e(t,\p)\right|+|I_6^\e(t,\p)|+\left|I_7^\e(t,\p)\right|+\left|I_8^\e(t,\p)\right|+\left|I_9^\e(t,\p)\right|\leq C\e\left(\|\p\|_{H^1(\Omega_A^\e)}+\|\p\|_{L^2(\Omega_B^\e}\right).
\end{align}
Finally, merging the individual estimates for the error terms (namely, estimates~\eqref{estimate_i1},\eqref{estimate_i2},\eqref{estimate_i3},\eqref{estimate_i4},\eqref{estimate_i5}) and the estimate~\eqref{estimate_ucor}, we conlude
\begin{align}\label{estimate_mech}
\|U_{\err}^\e\|_{L^2(\Omega)^{3}}+\|\nabla U_{\cor}^\e\|_{L^2(\Omega_A^\e)^{3\times3}}+\e\left\|\nabla U_\cor^\e\right\|_{L^2(\Omega_A^\e)^{3\times3}}\leq C\left(\sqrt{\e}+\e+\left\|\Theta_{\err}\right\|_{L^2(\Omega)}\right).
\end{align}
\begin{remark}\label{rem:mech}
With estimate~\eqref{estimate_mech} it is clear that the error in the mechanical part inherits the convergence rate from the heat-error (at least if it is not faster than $\sqrt{\e}$).
\end{remark}
If, additionally, Assumptions (A2) are fulfilled, it is also possible to first differentiate equations~\eqref{corrector_eq:ma} and~\eqref{corrector_eq:mb} with respect to time and to then choose the test function $\partial_tU_\cor$ for the variational formulation of the arising system.
This way, we can get the following estimate 
\begin{multline}\label{estimate_mechtime1}
\|\partial_tU_{\err}^\e\|_{L^2(\Omega)^{3}}+\|\nabla\partial_t U_{\cor}^\e\|_{L^2(\Omega_A^\e)^{3\times3}}+\e\left\|\nabla\partial_t U_\cor^\e\right\|_{L^2(\Omega_A^\e)^{3\times3}}\\
\leq C\left(\sqrt{\e}+\e+\left\|\Theta_{\err}\right\|_{L^2(\Omega)}+\left\|\partial_t\Theta_{\err}\right\|_{L^2(\Omega)}\right).
\end{multline}
Since this is done quite analogously to the estimates leading to inequality~\eqref{estimate_mech}, except for a few terms arising due to the time differentiation, the details of this are omitted.

If we take $\partial_t U_{\cor^\e}$ as a test function and follow the same strategy as in~\eqref{estimate_i1}, we can estimate
\begin{multline}\label{estimate_i1time}
\int_0^tI_1^\e(\tau,\partial_tU_{\cor0}^\e)\di{\tau}-\int_0^t\int_{\Omega_A^\e}\recon{\CC_A^{r}e_y(\tau^u)}\theta_A:e(\partial_tU_{\cor0}^\e)\di{x}\di{\tau}\\
\geq C_1\left(\left\|e(U_\cor^\e)(t)\right\|^2_{L^2(\Omega_A^\e)}-\left\|e(U_\cor^\e)(0)\right\|^2_{L^2(\Omega_A^\e)}\right)-C_2\int_0^t\left\|U_{\cor}^\e\right\|^2_{H^1(\Omega_A^\e)^3}\di{\tau}-C_3(\e+\e^2),
\end{multline}
where some integration by parts w.r.t.~time was done. 
Similarly, we obtain
\begin{multline}\label{estimate_i2time}
\int_0^tI_2^\e(\tau,\partial_tU_{\cor0}^\e)\di{\tau}
\geq C_1\e^2\left(\left\|e(U_\cor^\e)(t)\right\|^2_{L^2(\Omega_B^\e)^{3\times3}}-\left\|e(U_\cor^\e)(0)\right\|^2_{L^2(\Omega_B^\e)^{3\times3}}\right)\\-C_2\e^2\int_0^t\left\|e(U_{\cor}^\e)\right\|^2_{L^2(\Omega_B^\e)^{3\times3}}\di{\tau}-C_3\e^2.
\end{multline}
Moreover, it holds
\begin{align}\label{estimate_i5time}
\sum_{i=5}^9\left|I_i^\e(t,\partial_tU_{\cor0}^\e)\right|\leq C\e\left(\|U_\cor^\e\|_{H^1(\Omega_A^\e)^3}+\|U_\cor^\e\|_{L^2(\Omega_B^\e)^3}\right).
\end{align}
\subsection{Estimates for the heat conduction equations}\label{subsec:heat}
Now we go on with establishing some control on the error terms in the heat conduction equations.
Multiplying equations~\eqref{corrector_eq:ha} and~\eqref{corrector_eq:hb} with test functions $\p\in H_0^1(\Omega)$, integrating over $\Omega$, and then integrating by parts while using the interface conditions, we are lead to
\begin{multline*}
\underbrace{\int_{\Omega_A^\e}\partial_t\left(c_A^{r,\e}\Theta_A^\e-\kappa_A\int_{Y_A}c^{r}\di{y}\,\theta_A\right)\p\di{x}}_{=:E_1^\e(t,\p)}
+\underbrace{\int_{\Omega_A^\e}\left(K_A^{r,\e}\nabla\Theta_A^\e-\kappa_AK_A^{h}\nabla\theta_A\right)\cdot\nabla\p\di{x}}_{=:E_2^\e(t,\p)}\\
+\underbrace{\int_{\Omega_A^\e}\left(\left(c_A^{r,\e}\Theta_A^\e+\gamma_A^{r,\e}:\nabla U_A^\e\right)v^{r,\e}\right)\cdot\nabla\p\di{x}}_{=:E_3^\e(t,\p)}\\
+\underbrace{\int_{\Omega_A^\e}\left(K_A^{r,\e}\nabla\Theta_A^\e-\kappa_AK_A^{h}\nabla\theta_A\right)\cdot\nabla\p\di{x}
-\int_{\Gamma^\e}\kappa_AK_A^{h}\nabla\theta_A\cdot n^\e\p\di{x}}_{=:E_4^\e(t,\p)}\\
+\underbrace{\int_{\Omega_B^\e}\left(\partial_t\left(c_B^{r,\e}\Theta_B^\e\right)-\recon{\partial_t(c_B^{r}\Theta_B)}\right)\p\di{x}}_{=:E_6^\e(t,\p)}
+\underbrace{\int_{\Omega_B^\e}\left(\partial_t\left(\e\gamma_B^{r,\e}:\nabla U_B^\e\right)-\recon{\partial_t(\gamma_B^{r}:\nabla_y U_B)}\right)
\p\di{x}}_{=:E_6^\e(t,\p)}\\
+\underbrace{\int_{\Omega_B^\e}\left(c_B^{r,\e}\Theta_B^\e v^{r,\e}-\recon{c_B^{r}\Theta_B}\recon{v^{r}}\right)\cdot\nabla\p\di{x}}_{=:E_7^\e(t,\p)}
+\underbrace{\e\int_{\Omega_B^\e}\left(\gamma_B^{r,\e}:\nabla U_B^\e v^{r,\e}-\recon{\gamma^{r}_B}:\nabla\recon{U_B}\recon{v^{r}}\right)\cdot\nabla\p\di{x}}_{=:E_8^\e(t,\p)}\\
+\underbrace{\e^2\int_{\Omega_B^\e}\left(K_B^{r,\e}\nabla\Theta_B^\e-\recon{K_B^{r}}\nabla\recon{\Theta_B}\right)\cdot\nabla\p\di{x}}_{=:E_9^\e(t,\p)}\\
=\underbrace{-\int_{\Gamma^\e}W_\Gamma^{r,\e}\p\di{\sigma}+\int_{\Omega_A^\e}\kappa_AW_{\Gamma}^{h}\p\di{x}}_{=:E_{10}^\e(t,\p)}
+\underbrace{\int_{\Omega_A^\e}\kappa_A\int_{\Gamma}K_B^{r}\nabla_y\Theta_B\cdot n\di{s}\,\p\di{x}-\e^2\int_{\Gamma^\e}\recon{K_B^{r}}\nabla\recon{\Theta_B}\cdot n^\e\p\di{\sigma}}_{=:E_{11}^\e(t,\p)}\\
+\underbrace{\int_{\Omega_A^\e}\left(f_{\theta_A}^{r,\e}-\kappa_A\int_{Y_A}f_{\theta_A}^{r}\di{y}\right)\p\di{x}+\int_{\Omega_B^\e}\left(f_{\theta_B}^{r,\e}-\recon{f_{\theta_B}^{r}}\right)\p\di{x}}_{=:E_{12}^\e(t,\p)}+E_{13}^\e(t,\p)
\end{multline*}
where
$$
E_{13}^\e(t,\p)=\e\int_{\Omega_B^\e}\left(\recon{\dive_x\left(K_B^{r}\nabla\Theta_B\right)}+\dive\recon{K_B^{r}\nabla_x\Theta_B}\right)\p\di{x}.
$$
%
%
%
For the first term, we see that
\begin{multline*}
E_1^\e(t,\Theta_{\cor0}^\e)=\int_{\Omega_A^\e}\rho_Ac_{A}\partial_t\left(\Theta_A^\e(J^\e-\frac{1}{|Y_A(0)|}\left|Y_A\right|)\right)\Theta_{\cor0}^\e\di{x}\\
	+\int_{\Omega_A^\e}\rho_Ac_{A}\frac{1}{|Y_A(0)|}\partial_t\left(\left|Y_A\right|\Theta_\err^\e\right)\Theta_\err^\e\di{x}+R_2^\e(t),
\end{multline*}
where
\begin{equation*}
R_2^\e(t)=\e\int_{\Omega_A^\e}\rho_Ac_{A}\frac{1}{|Y_A(0)|}\partial_t\left(\left|Y_A\right|\Theta_\err^\e\right)m^\e\recon{\widetilde{\Theta}}.
\end{equation*}
Using the regularity estimates for $\widetilde{\Theta}$ and $\Theta_A^\e$, it is easy to see that there is a constant $c>0$ independent of $\e$ such that (for every $\delta>0$)
$$
\left|\int_0^t R_2^\e(\tau)\di{\tau}\right|\leq \int_0^t\left\|\Theta_\err^\e(\tau)\right\|^2\di{\tau}+\delta\left(\left\|\Theta_\err^\e(t)\right\|^2-\left\|\Theta_\err^\e(0)\right\|^2\right)+C_\delta t\e^2.
$$
With this estimate and Lemma~\ref{lemma:average}, we then get
\begin{multline}\label{estimate:E1}
\int_0^tE_1^\e(\tau,\Theta_{\cor0}^\e)\di{\tau}\\ \geq C_1\left(\left\|\Theta_\err^\e(t)\right\|_{\Omega_A^\e}^2-\left\|\Theta_\err^\e(0)\right\|_{\Omega_A^\e}^2\right)-C_2\left(\int_0^t\left\|\Theta_\err^\e(\tau)\right\|_{\Omega_A^\e}^2+\e\left\|\Theta_{\cor}^\e(\tau)\right\|_{H^1(\Omega_A^\e)}\di{\tau}+t\e^2\right).
\end{multline}
For the dissipation term of $\Omega_A^\e$, namely $E_2^\e$, we start by noticing that
\begin{align*}
\kappa_A\gamma_A^{h}:\nabla u_A+\kappa_A\int_{Y_A}\gamma_A^{r}:\nabla_y\tau^u\dy\,\theta_A=\kappa_A\int_{Y_A}\gamma_A^{r}:\left(\nabla u_A+\nabla_y\widetilde{U}\right)\di{y}
\end{align*}
and decompose
\begin{align*}
E_2^\e(t,\p)&=\int_{\Omega_A^\e}\partial_t\left(\gamma_A^{r,\e}:\nabla U_\cor^\e\right)\p\di{x}
	+\int_{\Omega_A^\e}\partial_t\left(\left(\gamma_A^{r,\e}-\recon{\gamma_A^{r}}\right):\nabla(U_A^\e-U_\cor^\e)\right)\p\di{x}\\
	&\quad+\int_{\Omega_A^\e}\partial_t\left(\recon{\gamma_A^{r}}:\nabla(U_A^\e-U_\cor^\e)-\kappa_A\int_{Y_A}\gamma_A^{r}:\left(\nabla u_A+\nabla_y\widetilde{U}\right)\di{y}\right)\p\di{x}.
\end{align*}
Applying Lemma~\ref{lemma:average} to 
$$
f=\partial_t\left(\gamma_A^{r}:\left(\nabla U_A+\nabla_y\widetilde{U}\right)-\kappa_A\int_{Y_A}\gamma_A^{r}:\left(\nabla u_A+\nabla_y\widetilde{U}\right)\di{y}\right),
$$
leads to
\begin{align*}
\left|E_2^\e(t,\p)\right|\leq C\left(\left(\|\nabla U_\cor\|_{L^2(\Omega_A^\e)^{3\times3}}+\|\nabla\partial_tU_\cor\|_{L^2(\Omega_A^\e)^{3\times3}}\right)\|\p\|_{L^2(\Omega_A^\e)}+\e\|\p\|_{H^1(\Omega_A^\e)}\right).
\end{align*}
%
In the case of $E_3^\e$, the estimate  $\e^{-1}\|v^\e\|_{L^\infty}(\Omega)\leq C$ (see \eqref{s:estimate_movement}) implies
\begin{align}\label{estimate:E3}
\left|E_3^\e(\tau,\p)\right|\leq C\e\|\nabla\p\|_{L^2(\Omega_A^\e)}\quad\text{for all}\ \p\in H^1(\Omega).
\end{align}
For handling the heat conduction functional, 
\begin{multline*}
E_4^\e(t,\p)=\int_{\Omega_A^\e}K_A^{r,\e}\nabla\Theta_\cor^\e\cdot\nabla\p\di{x}\\
	+\int_{\Omega_A^\e}\left(K_A^{r,\e}\nabla(\Theta_A^\e-\Theta_\cor^\e)-\frac{1}{|Y_A(0)|}K_A^{h}\nabla\Theta_A\right)\cdot\nabla\p\di{x}\\
	-\int_{\Gamma^\e}\frac{1}{|Y_A(0)|}K_A^{h}\nabla\Theta_A\cdot n^\e\p\di{x},
\end{multline*}
the strategy is exactly the same as with dealing with the $I_1^\e$-estimate of the mechanical part, see estimate~\eqref{estimate_i1}, which then leads to
\begin{equation}\label{estimate:E4}
E_4^\e(t,\Theta_{\cor0}^\e)\geq C_1\left\|\nabla\Theta_\cor^\e\right\|^2_{L^2(\Omega_A^\e)}-C_2\e(\left\|\Theta_{\cor0}^\e\right\|_{H^1(\Omega_A^\e)}+1+\e).
\end{equation}
Now, turning our attention to the next to functionals, $E_5^\e$, it follows easily from Lemma~\ref{lemma:estimates_trafo} that
\begin{align}\label{estimate:E5}
\int_0^tE_5^\e(\tau,\Theta_{\cor0}^\e)\di{\tau}&\geq C_1\left(\left\|\Theta_\err^\e(t)\right\|^2_{L^2(\Omega_B^\e)}
	-\|\Theta_\err^\e(0)\|_{L^2(\Omega_B^\e)}^2\right)-\int_0^t\left\|\Theta_\err^\e(\tau)\right\|_{L^2(\Omega_B^\e)}^2\di{\tau}-\e^2.
\end{align}
Estimates for the dissipation error terms, $E_5^\e$-$E_7^\e$, are given by
\begin{align}
|E_6^\e(t,\p)|&\leq C\e\left\|\p\right\|_{L^2(\Omega_B^\e)}\left(\left\|\nabla U_\err^\e\right\|_{L^2(\Omega_B^\e)^{3\times3}}+\left\|\nabla\partial_tU_\err^\e\right\|_{L^2(\Omega_B^\e)^{3\times3}}+\e\right),\label{estimate:E6}\\
|E_7^\e(t,\p)|&\leq C\e\left\|\nabla\p\right\|_{L^2(\Omega_B^\e)}\left(\left\|\Theta_\err^\e\right\|_{L^2(\Omega_B^\e)}+\e\right),\label{estimate:E7}\\
|E_8^\e(t,\p)|&\leq C\e^2\left\|\nabla\p\right\|_{L^2(\Omega_B^\e)}\left(\left\|\nabla U_\err^\e\right\|_{L^2(\Omega_B^\e)^{3\times3}}+\e\right)\label{estimate:E8}.
\end{align}
Similarly, we obtain
\begin{align}\label{estimate:E9}
E_9^\e(t,\Theta_{\cor0}^\e)\geq C_1\e^2\left\|\nabla\Theta_\err^\e\right\|^2_{L^2(\Omega_B^\e)}-C_2\e^2.
\end{align}
Taking a look at the interface velocity terms, we get
\begin{align*}
\left|E_{10}^\e(t,\p)\right|\leq\left|\int_{\Gamma^\e}\left(W_\Gamma^{r,\e}-\e\recon{W_\Gamma^{r}}\p\di{\sigma}\right)\right|
+\left|\int_{\Gamma^\e}\e\recon{W_\Gamma^{r}}\p\di{\sigma}+\int_{\Omega_A^\e}\kappa_AW_{\Gamma}^{h}\p\di{x}\right|.
\end{align*}
Using Lemma~\ref{lemma:estimates_trafo} for the functional $E_{11}^\e$, cf.~\cite{MV13}, the estimates on the functions that are involved, and our assumptions on the data, it is straightforward to show 
\begin{align}\label{estimate:E10}
\left|E_{10}^\e(t,\p)\right|+\left|E_{11}^\e(t,\p)\right|+\left|E_{12}^\e(t,\p)\right|\leq C\e\left(\|p\|_{H^1(\Omega_A^\e)}+\|\p\|_{L^2(\Omega_B^\e)}\right).
\end{align}
Finally, for the functional $E_{13}^\e$ catching some of ther terms arising in the elliptic part for $\Theta_{\err}$, we get
\begin{align}\label{estimate:E12}
\left|E_{13}^\e(t,\Theta_{\cor0})\right|\leq C\e\left(\left\|\Theta_{\err}\right\|+\e\right).
\end{align}
%
%
Summarizing those estimates~\eqref{estimate:E1}-\eqref{estimate:E12} and using Young's and Gronwall's inequalities, we arrive at
\begin{multline}\label{estimate_heat}
\|\Theta_\err^\e\|_{L^\infty(S\times\Omega)}+\|\nabla\Theta_\cor^\e\|_{L^2(S\times\Omega_A^\e)^3}+\e\|\nabla\Theta_\err\|_{L^2(S\times\Omega_B^\e)^3}\\
\leq C\big(\|\nabla U_\cor\|_{L^2(S\times\Omega_A^\e)^{3\times3}}+\e\|\nabla U_\err\|_{L^2(S\times\Omega_B^\e)^{3\times3}}\\ +\|\nabla \partial_tU_\cor\|_{L^2(S\times\Omega_A^\e)^{3\times3}}+\e\|\nabla \partial_tU_\err\|_{L^2(S\times\Omega_B^\e)^{3\times3}}+\sqrt{\e}+\e\big).
\end{multline}
\begin{remark}\label{rem:heat}
With~\eqref{estimate_heat} at hand, we conclude that estimates for $\nabla U_\cor^\e$ and $\nabla\partial_tU_\cor^\e$ will also lead to corresponding corrector estimates for the heat part.
\end{remark}

\subsection{Overall estimates}\label{subsec:overall}
Here, we combine the estimates from the preceding sections, Section~\ref{subsec:momentum} and Section~\ref{subsec:heat}.
It is clear that the following statement now follows directly from estimates~\eqref{estimate_mech} and~\eqref{estimate_heat}.
\begin{theorem}[Corrector for Weakly Coupled Problem]\label{cor1}
If we reduce our problem to a weakly coupled problem, that is, if we assume either $\alpha_A=\alpha_B=0$ (together with (A1)) or $\gamma_A=\gamma_B=0$ (together with (A2)), we have the following corrector estimate:
\begin{multline*}
\|\Theta_\err^\e\|_{L^\infty(S\times\Omega)}+\|U_{\err}^\e\|_{L^\infty(S;L^2(\Omega)^{3}}+\|\nabla\Theta_\cor^\e\|_{L^2(S\times\Omega_A^\e)^3}+\|\nabla U_{\cor}^\e\|_{L^\infty(S;L^2(\Omega_A^\e))^{3\times3}}\\
+\e\|\nabla\Theta_\err\|_{L^2(S\times\Omega_B^\e)^3}+\e\left\|\nabla U_\cor^\e\right\|_{L^\infty(S;L^2(\Omega_B^\e))^{3\times3}}\leq C (\sqrt{\e}+\e).
\end{multline*}
\end{theorem}

Moreover, for the heat part, we take $\Theta_{\cor0}$ and for the mechanical part $\partial_t U_{\cor0}^\e$ as a test function, sum the weak formulations, integrate over $(0,t)$ and get
\begin{multline}
\int_0^t\left(\sum_{i=1}^6I_i^\e(\tau,\partial_tU_{\cor0}^\e)+\sum_{i=1}^9E_i^\e(\tau,\Theta_{\cor0}^\e) \right)\di{\tau}\\
=\int_0^t\left(\sum_{i=7}^9I_i^\e(\tau,\partial_tU_{\cor0}^\e)+\sum_{i=10}^{13}E_i^\e(\tau,\Theta_{\cor0}^\e) \right)\di{\tau}.
\end{multline}
Now, we first take a view on the error terms corresponding to the coupling terms for the $\Omega_B^\e$ part for both the mechanical and the heat part, namely $I_4^\e$, $E_6^\e$, and $E_8^\e$.
While $E_8^\e$ can be controlled in terms of $\e\nabla U_\err$ and $\e\nabla\Theta_\err$ (see inequality~\eqref{estimate:E8}), this is not possible for either $I_4^\e$ or $E_6^\e$ due to the involved time derivatives.
If we take a look at the sum of those (appropriately scaled)\footnote{Assuming $\alpha_B\neq0$.} two terms, however, we see that they counterbalance each other leading to 
\begin{align}
\left|\frac{\gamma_B}{\alpha_B}I_4^\e(\tau,\partial_tU_{\cor0}^\e)+E_6^\e(\tau,\Theta_{\cor0}^\e)\right|\leq C\|\Theta_\err^\e\|^2_{L^2(\Omega_B^\e)}+\e^2\|\nabla U_\cor\|^2_{L^2(\Omega_B^\e)}+\e^2.
\end{align}
Note that with estimate~\eqref{estimate_i2time}, the $\e^2\|\nabla U_\cor\|^2_{L^2(\Omega_B^\e)}$ is resolvable via Gronwall's inequality.

This, unfortunately, does not work for the coupling parts in $\Omega_A^\e$ : Here, we would have to apply Lemma~\ref{lemma:average} at the cost of additional derivatives (we only get control in $H^1$ and not in $L^2$), which, in general, can not be compensated without structural assumptions.

As a result of this observation and the estimates collected in the previous sections, we get:

\begin{theorem}[Corrector for Microscale Coupled Problem]\label{cor2}
If we simplify our problem so that there is only coupling in the $\Omega_B^\e$ part, that is, if we assume $\alpha_A=\gamma_A=0$, we have the following corrector estimate:
\begin{multline*}
\|\Theta_\err^\e\|_{L^\infty(S\times\Omega)}+\|U_{\err}^\e\|_{L^\infty(S;L^2(\Omega)^{3}}+\|\nabla\Theta_\cor^\e\|_{L^2(S\times\Omega_A^\e)^3}+\|\nabla U_{\cor}^\e\|_{L^\infty(S;L^2(\Omega_A^\e))^{3\times3}}\\
+\e\|\nabla\Theta_\err\|_{L^2(S\times\Omega_B^\e)^3}+\e\left\|\nabla U_\cor^\e\right\|_{L^\infty(S;L^2(\Omega_B^\e))^{3\times3}}\leq C (\sqrt{\e}+\e).
\end{multline*}
\end{theorem}

\section*{Acknowledgments}
The authors are indebted to Michael B\"ohm (Bremen) for initiating and supporting this research.
AM thanks NWO MPE ``Theoretical estimates of heat losses in geothermal wells'' (grant nr. 657.014.004) for funding.

\bibliography{literature}{}
\bibliographystyle{plain}

\medskip
\medskip

\end{document}